\documentclass[draft]{amsart}
\usepackage{amssymb} % SRC Specials: DVI [Inverse] Search
\newtheorem{theorem}{Theorem}[section]
\newtheorem{lemma}[theorem]{Lemma}

\theoremstyle{definition}

\newtheorem{remark}[theorem]{Remark}
\theoremstyle{approach}

\numberwithin{equation}{section}

\begin{document}
	\setcounter{page}{1}
	\title{The Bochner-Schoenberg-Eberlein Property for 
		Fr\'echet C$^*$-algebras and uniform Fr\'echet algebras}
	\author[M. Amiri and  A. Rejali]{M. Amiri and A. Rejali}

\begin{abstract}
	
Takahasi and Hatori introduced a class of commutative Banach algebras 
which satisfy a Bochner-Schoenberg-Eberlein-type inequality. Baised on their
results we introduced a class of commutative Fr\'echet algebras which satisfy this 
property. We show that Fr\'echet C$^*$-algebras and uniform Fr\'echet algebras
are {\text BSE}-algebras.
\end{abstract}

\maketitle \setcounter{section}{0}

\section{Introduction and preliminaries}

The locally C$^*$-algebras were studied in \cite{9, 21} and with various
 names in \cite{2, 4, 5, 6, 16} and elsewhere. These objects are called
 LMC$^*$-algebras and 
pro-C$^*$-algebras  in \cite{16} and \cite{20, 15}, respectively. If the topology is determined only by
 a countable number of C$^*$-seminorms, then such an algebra is called $\sigma$-C$^*$-algebra 
 or F$^*$-algebras; see \cite{2} and \cite{3}. 
Countable inverse limits of C$^*$-algebras
 were introduced in \cite{3} under the name F$^*$-algebras. 
  Takahasi and Hatori showed that 
 a semisimple type I-{\text BSE}-algebra with a bounded approximate
 identity is isomorphic to a commutative C$^*$-algebra and vice versa;
 see \cite[Theorem 3]{17}.
 Recently in \cite{1}, we studied the notion of {\text BSE}-algebras in the
 class of Fr\'echet algebras, and generalized
 some important results concerning this notion from Banach algebras to the Fr\'echet case.
 In addition, uniform Fr\'echet algebras and uniform Banach algebras were studied in \cite{7}.
 In this paper, we show that each commutative Fr\'echet C$^*$-algebra and 
uniform Fr\'echet algebra are a {\text BSE}-algebra.

Before proceeding to the main results, let us present some basic definitions and frameworks.  
 
A Fr\'echet algebra is a complete topological algebra where 
its topology is generated by a countable family of increasing
submultiplicative seminorms. In this definition, the underlying
Fr\'echet space is supposed to be Hausdorff. 
Many authors studied about Fr\'echet algebras; see for example \cite{ARR4,7,M,Rahnama111,RR0}. 
 Rejali and Alimohammadi \cite{3A} also provided a survey paper of results on the Fr\'echet algebras.
 
Following \cite{3}, a Fr\'echet $*$-algebra 
 is a Fr\'echet algebra with a continuous
 involution. Let $\mathcal A$ be a Fr\'echet $*$-algebra with identity $e$. Then, there exists a 
 sequence $(p_n)_{n\in\mathbb{N}}$ of seminorms for $\mathcal A$ such that
 $p_{n}(e)=1$ and $p_{n}(a^*)=p_{n}(a)$ for all $n\in \mathbb{N}$ and $a\in \mathcal A$. 
 Such a sequence is called a $*$-sequence of seminorms for $\mathcal A$. 
The Fr\'echet $C^{\ast}$-algerbra $(\mathcal A,p_{n})_{n\in \mathbb{N}}$ is a Fr\'echet $*$-algebra
which has the C$^*$-property $p_n(a^{\ast}a)=(p_n(a))^{2}$, for each $a\in\mathcal A$ and $n \in \mathbb{N}$.
In this case, $(p_n)_{n\in\mathbb{N}}$
is called an F$^*$-sequence
of seminorms. The typical structures 
indicate that every Fr\'echet $*$-algebra and Fr\'echet C$^*$-algebra are
inverse limit of Banach $*$-algebras and Banach C$^*$-algebras, respectively; see \cite{14} for details. 

Let $\mathcal A$ be a pro-C$^*$-algebra.
Following \cite{15}, $\mathcal A$ is a complete Hausdorff topological $*$-algebra 
over $\mathbb{C}$ whose topology is determined by its continuous 
 C$^*$-seminorms in the sense that a net $(a_{\lambda})_{\lambda}$ converges to
 zero in $\mathcal{A}$ if and only if $p(a_{\lambda})\longrightarrow_{\lambda} 0$ for every continuous 
 C$^*$-seminorm $p$ on $\mathcal A$. 
Let now $S(\mathcal A)$ denotes the set of all
 continuous C$^*$-seminorms on $\mathcal A$,
directed by the order $p\leq q$ if
 $p(a) \leq q(a)$ for all $a\in \mathcal A$.
For $p\in S(\mathcal A)$, let ${\mathcal A}_p$
denotes the completion of $\mathcal A/ \text{ker}(p)$ in the norm given
 by $p$. Then, ${\mathcal A}_p$ is a C$^*$-algebra; see \cite{15} for more details. 
 Note that there is a canonical surjective map
 ${\mathcal A}_q \rightarrow {\mathcal A}_p$ whenever $p\leq q$.
 A topological $*$-algebra $\mathcal A$ is a pro-C$^*$-algebra
 if and only if it is inverse limit of an inverse 
 system of C$^*$-algebras and $*$-homomorphisms. Indeed, by applying \cite [Proposition 1.2]{15},
$$
 \mathcal A\cong \underset{\underset{p\in S(\mathcal A)}{\longleftarrow}}{lim}{\mathcal A}_{p}.
$$
 For definitions of inverse and direct limits of algebras, we refer the reader to \cite{111}.

Following \cite{2,15}, any C$^*$-algebra 
and especially any Fr\'echet C$^*$-algebra is a pro-C$^*$-algebra.

Let $\mathcal A$ be a commutative semisimple Banach algebra, and let $\Delta(\mathcal A)$ denotes the set of all nonzero multiplicative linear functionals on $\mathcal A$.
For $a\in\mathcal{A}$, define
$$
\widehat{a}:\Delta(\mathcal A)\rightarrow\mathbb{C},\;\;\;\;\widehat{a}(\varphi)=\varphi(a)\;\;\;\;\;\;\;\;(\varphi\in\Delta(\mathcal A))
$$
Let $\widehat{\mathcal A}=\lbrace\widehat{a}: a\in \mathcal A\rbrace$
 and
 $\Phi:\Delta(\mathcal A)\rightarrow \mathbb{C}$ be a continuous
 function such that $\Phi \cdot \widehat{\mathcal A} 
\subseteq \widehat{\mathcal A}$.
Suppose that
$$\mathcal{M}(\mathcal A)=\big{\lbrace} \Phi: \Delta(\mathcal A)\rightarrow
 \mathbb{C}: \text{$\Phi$ is continuous and}\; \Phi \cdot \widehat{\mathcal A}
 \subseteq \widehat{\mathcal A}\big{\rbrace}.$$
Following \cite{13}, 
$\mathcal{M}(\mathcal A)=\widehat{\mathsf{M}(\mathcal A)}$
where $\mathsf{M}(\mathcal A)$ is the set 
of all multipliers on $\mathcal A$, i.e.
$$
\mathsf{M}(\mathcal A)=\big{\{}T:\mathcal
A\rightarrow\mathcal A: \text{for all $a,b\in\mathcal{A}$},\; T(a)b=aT(b)\big{\}}.
$$
We also recall from \cite{17,18} that
a bounded continuous function $\sigma$ on $\Delta(\mathcal A)$ is called
 a {\text BSE}-function if
there exists a positive real number $\beta$ such that the inequality
$$\vert\sum_{i=1}^{n}c_{i}\sigma (\varphi_{i})\vert\leq\beta
\Vert\sum_{i=1}^{n}c_{i}\varphi_{i}\Vert_{\mathcal A^{\ast}}$$
holds, for every
finite number of complex-numbers $c_1,\cdots,c_n$ and the same number of
 $\varphi_1,\cdots,\varphi_n$ in $\Delta(\mathcal A)$. 
 In \cite{1}, we generalized this as follows.
 
 Let $(\mathcal A,p_n)_{n\in\mathbb{N}}$ be a commutative semisimple Fr\'echet algebra. Consider $\mathcal{A}^*$ 
 the topological dual of $\mathcal A$. Following \cite{M}, the strong topology on $\mathcal{A}^*$
 is generated by seminorms $(P_M)$ where $M$ is a bounded set in $\mathcal A$.
 Moreover $\Delta(\mathcal A)$, $\mathsf{M}(\mathcal A)$ and $\mathcal{M}(\mathcal A)$
 are defined similar to the Banach case; see \cite{7,15} for details.
 Following \cite{1}, a bounded complex-valued
continuous function $\sigma$ on $\Delta(\mathcal A)$, is called a
{\text BSE}-function if there exist a bounded set $M$ in $\mathcal A$ and 
 a positive real number $\beta_{M}$ such that for every finite number
 of complex-numbers $c_1,\cdots, c_n$ and the same number of 
 $\varphi_1,\cdots, \varphi_n$ in $\Delta(\mathcal A)$ the inequality
$$\vert\sum_{i=1}^{n} c_{i}\sigma(\varphi_{i})\vert\leq\beta_{M}
P_{M}(\sum_{i=1}^{n}c_{i}\varphi_{i})$$
holds. 
The set of all {\text BSE}-functions is denoted by $C_{\text{BSE}}(\Delta(\mathcal A))$.
We showed that $C_{\text{BSE}}(\Delta(\mathcal A))$ is also a commutative semisimple Fr\'echet algebra.
Following \cite{1} and
by applying similar arguments of \cite[Theorem 1.2.2]{13}, 
for each $T\in \mathsf{M}(\mathcal A)$, there exists a unique continuous and bounded function
$\widehat{T}$ on $\Delta(\mathcal A)$ such that
$$
\varphi(Ta)=\widehat{T}(\varphi)\varphi(a)\;\;\;\;\;\;\;\;\;\;\;\;\big{(}a\in\mathcal A,\;\varphi \in \Delta(\mathcal A)\big{)}.
$$
Set
$\widehat{\mathsf{M}(\mathcal A)}=\lbrace \widehat{T}:T \in \mathsf{M}(\mathcal A)\rbrace$.
If $\widehat{\mathsf{M}(\mathcal A)}=C_{\text{BSE}}(\Delta(\mathcal A))$, then
$\mathcal A$ is called a {\text BSE}-Fr\'echet algebra.

\section{Main results}

Let $(\mathcal{A},p_{n})_{n\in\mathbb{N}}$
be a Fr\'echet C$^*$-algebra. As mentioned in the previous section,
$\mathcal A\cong \underset{\longleftarrow}{lim} {\mathcal A}_n$
where each map ${\mathcal A}_{n+1} 
\rightarrow {\mathcal A}_n$ is surjective. 
The canonical map 
$\mathcal A \rightarrow {\mathcal A}_n$ is also 
surjective; see \cite[Lemma 5.1]{15} and \cite[II.5]{111}.
Now, assume that
$\varphi:\mathcal A \rightarrow \mathcal B$ 
is a surjective homomorphism of Fr\'echet C$^*$-algebras. If $\mathcal A$ has 
a countable approximate identity, then we can define a surjective map 
$\underset{\longleftarrow}{lim} \mathsf{M}({\mathcal A}_n)\rightarrow\underset{\longleftarrow}{lim} \mathsf{M}({\mathcal B}_n)$ and so a surjective map $\mathsf{M}(\mathcal A) \rightarrow \mathsf{M}(\mathcal B)$; see \cite [Theorem 5.11]{15} for more details.
We also recall from \cite[proposition 5.9]{15} that
the multiplier algebra of Fr\'echet C$^*$-algebras is a 
 Fr\'echet C$^*$-algebra. In fact, $\mathsf{M}(\underset{\longleftarrow}
{lim}{\mathcal A}_n)\cong \underset {\longleftarrow}{lim}\mathsf{M}({\mathcal A}_n).$ 
Moreover, as we pointed out earlier, the map ${\mathcal A}_{n+1} \rightarrow {\mathcal A}_n$ is assumed surjective. However,
$\mathsf{M}({\mathcal A}_{n+1})\rightarrow \mathsf{M}({\mathcal A}_n)$ 
need not be surjective. 

\begin{remark}\label{2}
Let $(\mathcal A,p_n)_{n\in\mathbb{N}}$
be a commutative semisimple unital Fr\'echet C$^{\ast}$-algebra.
Following \cite[page 171]{15} and also \cite[Theorem 3.14]{15}, there exists a sequence of Banach $C^*$-algebras $({\mathcal A}_n)$,
 such that
\begin{enumerate}
\item[(i)] 
$\mathcal A\cong\underset{\longleftarrow}{lim}{\mathcal A}_n;$
\item[(ii)] 
$\Delta(\mathcal A)\cong\underset{\longrightarrow}{lim}
\Delta({\mathcal A}_n);$ 
\item[(iii)] 
$\mathcal{M}(\mathcal A)\cong
\underset{\longleftarrow}{lim}{\mathcal{M}(\mathcal A_{n})}.$ 
\end{enumerate}
\end{remark}

\begin{lemma} \label{1}
Let $(\mathcal A,p_n)_{n\in\mathbb{N}}$ be a commutative semisimple unital Fr\'echet $C^*$- 
algebra. Then, there exists a sequence of Banach C$^{\ast}$-algebras $({\mathcal A}_n)$,
 such that 
$$
C_{\text{BSE}}(\Delta(\mathcal A))\cong
\underset{\longleftarrow}{lim}C_{\text{BSE}}(\Delta(\mathcal A_{n})).$$ 
\end{lemma}

\begin{proof}
For each $n\in\mathbb{N}$, let
$I_n=\lbrace a\in \mathcal A:p_n(a)=0\rbrace.$
 Following \cite[page 77]{7}, since $(p_n)_{n}$ is increasing, for each $m,n\in\mathbb{N}$ with $n \leq m$,
we have $I_m\subseteq I_n$ and also
$$\big{\Vert} a+I_n \big{\Vert}\leq \big{\Vert} a+I_m \big{\Vert}\;\;\;\;\;\;\;\;\;(a\in\mathcal{A}).$$
Let $m,n\in\mathbb{N}$ with $n \leq m$.
We consider the map 
$$\rho_{mn}:\mathcal A_m \rightarrow \mathcal A_n,\;\;\;\;
 a+I_m \mapsto a+I_n.$$ 
Then, 
$$\big{\Vert} \rho_{mn}(a+I_m) \big{\Vert}=
 \big{\Vert}a+I_n\big{\Vert}\leq \big{\Vert} a+I_m\big{\Vert},$$
and consequently, $\Vert \rho_{mn} \Vert \leq 1$. 
We also consider the adjoint mapping
$$\rho_{mn}^*:\mathcal A_n^*\rightarrow \mathcal A_m^*,\;\;\;\; f \mapsto
 \overline{f}$$ 
 where $\overline{f}=f\circ\rho_{mn}$.
 Therefore,
 $\big{\Vert} \overline{f}\big{\Vert}_{\mathcal A_m^*} \leq\big{\Vert} 
 f\big{\Vert}_{\mathcal A_n^*}$.
 We now define
 $$\pi_{mn}:C_{\text{BSE}}(\Delta(\mathcal A_m))\rightarrow
 C_{\text{BSE}}(\Delta(\mathcal A_n)),\;\;\;\;\sigma \mapsto \overline{\sigma}$$
by
$\overline{\sigma}(\psi):=\sigma(\psi \circ \rho_{mn})$ for all $\psi \in \Delta(\mathcal A_n)$.
To prove that $\pi_{mn}$ is well-defined, let $\sigma \in C_{\text{BSE}}(\Delta(\mathcal A_m))$.
We show that
$\overline{\sigma}\in C_{\text{BSE}}(\Delta(\mathcal A_n))$.
In accordance with the definition in \cite{17}, there exists a positive real number 
 $\beta$ such that for every finite number of complex-numbers 
 $c_{1},...,c_{k}$ and the same
 number of $\psi_{1},...,\psi_{k}$ in $\Delta(\mathcal A_m),$ the 
 inequality 
$$\big{\vert}\sum_{i=1}^{k} c_{i}\sigma(\psi_i)
\big{\vert} \leq \beta\; \big{\Vert} \sum_{i=1}^{k}c_{i}\psi_{i}\big{\Vert}_{\mathcal A_m^*}$$
holds. Therefore,
\begin{eqnarray*}
\big{\vert}\sum_{i=1}^{k} c_{i}\overline{\sigma}(\psi_i)\big{\vert}&=&\big{\vert} \sum_
{i=1}^{k} c_{i} \sigma(\psi_i \circ \rho_{mn})\big{\vert} \\
&\leq& \beta\; \big{\Vert} \sum_{i=1}^{k} c_{i}{(\psi_i o \rho_{mn})}
\big{\Vert}_{\mathcal A_m^*}\\
&=& \beta\; \big{\Vert} \sum_{i=1}^{k}c_{i}\overline{\psi_i}\big{\Vert}_{\mathcal A_m^*}\\ 
&=& \beta\; \big{\Vert}\big{(}\sum_{i=1}^{k}c_{i}\psi_i\overline{\big{)}}\big{\Vert}_{\mathcal A_m^*}\\
&\leq& \beta\; \big{\Vert} \sum_{i=1}^{k}c_{i} \psi_i \big{\Vert}_{\mathcal A_n^*}.
\end{eqnarray*}
Consequently, $\overline{\sigma} \in C_{\text{BSE}}(\Delta(\mathcal A_n))$.
Following \cite[II.5]{111}, we need to 
indicate that
\begin{enumerate}
\item[(i)]
for all $n \in \mathbb{N}$, $\pi_{nn}$ is the identity map on $C_{\text{BSE}}(\Delta(\mathcal A_n))$;
\item[(ii)]
$\pi_{mn}\circ\pi_{km}=\pi_{kn}$ for all $m,n,k\in\mathbb{N}$ with $n\leq m\leq k$.
\end{enumerate}
To prove (i), let $\sigma\in C_{\text{BSE}}(\Delta(\mathcal A_n))$.
Then, for each $\psi \in \Delta(\mathcal A_n)$, we have
$$\pi_{nn}(\sigma)(\psi)=\overline{\sigma}(\psi)=
\sigma(\psi \circ \rho_{nn})=\sigma(\psi).$$
Hence, $\pi_{nn}(\sigma)=\sigma$.
To prove (ii), consider $\sigma \in C_{\text{BSE}}(\Delta(\mathcal A_k))$.
Thus, 
$$\pi_{km}(\sigma)=\overline{\sigma}\in C_{\text{BSE}}(\Delta(\mathcal A_m)),$$
and so
for all $\psi \in \Delta(\mathcal A_n)$, we have
\begin{eqnarray*}
\pi_{mn}\circ\pi_{km}(\sigma)(\psi)&=&\pi_{mn}(\overline{\sigma})(\psi) \\
&=&\overline{\sigma}(\psi\circ\rho_{mn}) \\
&=&\sigma(\psi\circ\rho_{mn}\circ\rho_{km}) \\
&=&\sigma(\psi\circ\rho_{kn}).
\end{eqnarray*}
On the other hand, we have 
$$\pi_{kn}(\sigma)(\psi)=
\sigma(\psi \circ \rho_{kn}).$$
 Therefore, $\pi_{mn}\circ\pi_{km}
=\pi_{kn}$ for all $m,n,k\in\mathbb{N}$ with $n\leq m\leq k$.
Now, for each $n\in\mathbb{N}$, consider the map
$$
\rho_n:\mathcal A\rightarrow\mathcal A_n,\;\;\;\;a\mapsto a+ I_n.
$$
and define 
$$\pi_{n}:C_{\text{BSE}}(\Delta(\mathcal A)) \rightarrow
 C_{\text{BSE}}(\Delta(\mathcal A_n)),\;\;\;\;
\sigma \mapsto \widetilde{\sigma}_n$$
where $\widetilde{\sigma}_n(\psi)=\sigma(\psi\circ\rho_n)$ for all $\psi \in \Delta(\mathcal A_n)$.
For each $m,n\in\mathbb{N}$ with $n \leq m$, we have $\pi_{mn}\circ\pi_m=\pi_n$.
Indeed, for $\sigma\in C_{\text{BSE}}(\Delta(\mathcal A))$ and $\psi \in \Delta(\mathcal A_n)$ we have
\begin{eqnarray*}
\pi_{mn}\circ\pi_m(\sigma)(\psi)&=&
\pi_{mn}(\widetilde{\sigma}_m(\psi))\\
%\pi_m(\sigma)(\psi\circ\rho_{mn}) \\
&=&\widetilde{\sigma}_m(\psi\circ\rho_{mn}) \\
&=&\sigma(\psi\circ\rho_{mn}\circ\rho_m) \\
&=&\sigma(\psi\circ\rho_n) \\
&=&\widetilde{\sigma}_n(\psi) \\
&=&\pi_n(\sigma)(\psi).
\end{eqnarray*}
Let $\mathcal{B}$ be a Fr\'echet algebra
and consider each mapping
$\varphi_n:\mathcal{B} \rightarrow C_{\text{BSE}}(\Delta(\mathcal A_n))$
such that 
$\pi_{mn}\circ\varphi_m=\varphi_n$
for $m,n\in\mathbb{N}$ with $n\leq m$.
In this case,
 there exists a unique mapping
$\varphi:\mathcal{B} \rightarrow C_{\text{BSE}}(\Delta(\mathcal A))$ such that
 $\pi_n\circ\varphi=\varphi_n$ for all $n \in \mathbb N$. 
Moreover, $\Delta(\mathcal{A})=\bigcup_{n=1}^{\infty}\Delta(\mathcal{A}_n)$
and for each $m,n\in\mathbb{N}$ with $n\leq m$ we have $\Delta(\mathcal{A}_n)\subseteq\Delta(\mathcal{A}_m)$; see \cite[Theorem 3.2.8]{7}.
Therefore, for each $\psi \in \Delta(\mathcal A)$, 
there exists an $n\in \mathbb N$ such that $\psi \in \Delta(\mathcal A_n)$.
By applying part (ii) of Remark \ref{2}, suppose 
$\Delta(\mathcal A_n)\rightarrow \Delta(\mathcal A)$, $\psi\mapsto\overline{\psi}$
such that
$\psi(a+I_n)=\overline{\psi}(a)$ for all $a\in \mathcal A$.
We now define
$\varphi(b)(\overline{\psi}):=\varphi_n(b)(\psi)$ for all $\psi \in
 \Delta(\mathcal A_n)$ and $b\in\mathcal{B}$.
 Since
 $$
 \psi\circ\rho_{n}(a)=\psi(a+I_n)=\overline{\psi}(a)\;\;\;\;\;\;\;(a\in\mathcal{A})
 $$
and
 $$\overline{\psi\circ\rho_{mn}}(a)=\psi\circ\rho_{mn}(a+I_m)=
\psi(a+I_n)=\overline{\psi}(a)\;\;\;\;\;\;\;(a\in\mathcal{A}),$$
we have
$$\overline{\psi}=\psi\circ\rho_{n}=\overline{\psi\circ\rho_{mn}},$$
for $m\geq n$. 
Hence, for $b\in\mathcal{B}$ and $m\geq n$, we have
\begin{eqnarray*}
\varphi_{n}(b)(\psi)&=&\pi_{mn}\circ\varphi_{m}(b)(\psi) \\
&=&\varphi_{m}(b)(\psi\circ\rho_{mn}) \\
&=&\varphi(b)(\overline{\psi\circ\rho_{mn}}) \\
&=&\varphi(b)(\psi\circ\rho_n) \\
&=&\pi_{n}\varphi(b)(\psi).
\end{eqnarray*}
Accordingly, $\varphi_{n}=\pi_{mn}\,o\,\varphi_{m}=\pi_{n}\,o\,\varphi$ and by definition, 
$$C_{\text{BSE}}(\Delta(\mathcal A))\cong
\underset{\longleftarrow}{lim}\, C_{\text{BSE}}(\Delta(\mathcal A_{n})).$$
\end{proof}

The following result is now immediate.
  
\begin{theorem}
Let $\mathcal A$ be a commutative semisimple unital Fr\'echet C$^{\ast}$-algebra. Then, 
\begin{enumerate}
\item[(i)] $\mathcal A$ is a \text{BSE}-algebra.

\item[(ii)] $\mathcal A
\cong \underset{\longleftarrow}{lim}  C_0(\Delta(\mathcal A_n)).$

\end{enumerate}
\end{theorem}

\begin{proof}
(i) Since $\mathcal A\cong\underset{\longleftarrow}{lim} \mathcal A_n$ and 
${\mathcal A}_n$ is a commutative Banach $C^{\ast}$-algebra, by using \cite[Theorem 3]{17}, it
 is a \text{BSE}-algebra and is semisimple. Thus, 
$$\mathcal M({\mathcal A}_n)=\widehat{\mathsf{M}({\mathcal A}_n)}=C_{\text{BSE}}(\Delta({\mathcal A}_n)).$$
Therefore, by applying Remark \ref{2} and Lemma \ref{1}, we have
$$\mathcal M({\mathcal A})\cong\underset{\longleftarrow}{lim} \mathcal{M}(\mathcal A_{n})\cong
\underset{\longleftarrow}{lim} C_{\text{BSE}}(\Delta({\mathcal A}_n))\cong
C_{\text{BSE}}(\Delta(\mathcal A)).$$ 
Indeed, $\mathcal A$ is semisimple. Hence, $\widehat {\mathsf{M}(\mathcal A)}=C_{\text{BSE}}(\Delta(\mathcal A))$ and $\mathcal A$ is a $\text{BSE}$-algebra.

(ii) Let each $\mathcal A_n$ be a $C^{\ast}$-Banach algebra. Then for each $n$, $\mathcal A_n \cong C_0(\Delta(\mathcal A_n))$
 by Gelfand theorem \cite[ Theorem 3] {17}. Consequently, 
$$\underset{\longleftarrow}{lim} \mathcal A_n\cong\underset{\longleftarrow}{lim}
C_0(\Delta(\mathcal A_n)).$$ Thus, 
$\mathcal A \cong \underset{\longleftarrow}{lim}
C_0(\Delta(\mathcal A_n))$.
\end{proof}

The Banach algebra $(\mathcal B,\Vert . \Vert)$ is called a unifom Banach 
algebra if $\Vert a\Vert^2=\Vert a^2 \Vert$ for all $a\in \mathcal B.$
 In this case, we have 
 $$\Vert a \Vert=\Vert \widehat{a} \Vert_{\Delta(\mathcal B)}=
 sup\lbrace \vert \widehat{a}(\varphi)\vert:\varphi\in \Delta(\mathcal B)\rbrace$$ 
 for all $a\in \mathcal B$. Furthermore,
$\Gamma: \mathcal B \rightarrow \widehat{\mathcal B}$, $a \mapsto \widehat{a}$ 
is a topological homomorphism and so $(\widehat{\mathcal B},\Vert . \Vert_{\Delta(\mathcal B)})$ 
is a uniform Banach algebra; see \cite[4.1.1]{7}. In addition, 
each uniform Banach algebra is topologically and
 algebraically isomorphic to a closed point separating subalgebra of $C(K)$ 
 for some compact nonempty Hausdorff space $K$. 
Uniform Fr\'echet algerbras are introduced similar to the Banach case.
Indeed, the Fr\'echet algebra
 $(\mathcal A,p_{n})_{n\in\mathbb{N}}$ is called uniform Fr\'echet algerbra if $p_n(a^2)=(p_n(a))^2$ 
 for all $a\in \mathcal A$ and $n\in \mathbb{N}$; see \cite[Definition 4.1.2]{7}.

\begin{lemma}\label{3}
Let $(\mathcal A,p_n)_{n\in\mathbb{N}}$ be a commutative semisimple uniform Fr\'echet algebra.
 Then, there exists a sequence $(\mathcal A_n)$ of uniform Banach algebras such that
\begin{enumerate}
\item[(i)] $\mathcal A\cong\underset{\longleftarrow}{lim} \mathcal A_n;$

\item[(ii)] $\Delta(\mathcal A)\cong\underset{\longrightarrow}{lim}
\Delta(\mathcal A_n);$

\item[(ii)] $C_{\text{BSE}}(\Delta(\mathcal A))\cong
\underset{\longleftarrow}{lim}C_{\text{BSE}}(\Delta(\mathcal A_{n}));$

\item[(iv)] $\mathcal{M}(\mathcal A)\cong
\underset{\longleftarrow}{lim}{\mathcal{M}(\mathcal A_{n})}.$ 
\end{enumerate}
\end{lemma}

\begin{proof}
It is evident that (i) and (ii) are valid by applying 
\cite[Theorem 4.1.3]{7} and \cite[Theorem 3.2.8]{7}, respectively.  
Moreover, by a similar arguments as used in Lemma
 \ref{1} and Remark \ref{2}, (iii) and (iv) are established, respectively.
\end{proof}

\begin{theorem}
Any commutative semisimple uniform Fr\'echet algebra is a \text{BSE}-algebra.
\end{theorem}

\begin{proof}
Let $\mathcal{A}$ be a commutative semisimple uniform Fr\'echet algebra.
Then, $\mathcal A=\underset{\longleftarrow}{lim} \mathcal A_n$
where each $\mathcal A_n$
is a uniform Banach algebra.
By \cite[Theorem 2.6]{12}, 
each uniform Banach algebra is a
\text{BSE}-algebra. Hence, 
$$C_{\text{BSE}}(\Delta(\mathcal A_n))=\widehat{\mathsf{M}(\mathcal A_n)}=\mathcal M(\mathcal A_n)
$$
 for each $n \in \mathbb N$. Therefore, by Lemma \ref{3}, we have
 $$
C_{\text{BSE}}(\Delta(\mathcal A))=\mathcal M(\mathcal A),
$$
and so $\mathcal A$ is a \text{BSE}-algebra.
\end{proof}

\noindent {\bf Acknowledgment.} This research was partially supported by 
the center of excellence for mathematics at the University of Isfahan.

\vspace{9mm}

{\footnotesize \noindent

\noindent
M. Amiri\\
Department of Pure Mathematics,
University of Isfahan,
Isfahan, Iran\\
mitra.amiri@sci.ui.ac.ir\\ 
mitra75amiri@gmail.com\\

\noindent
A. Rejali\\
Department of Pure Mathematics,
University of Isfahan,
Isfahan, Iran\\
rejali@sci.ui.ac.ir\\

\end{document}